\documentclass{amsart}
\usepackage{amsmath,amssymb,epsfig,amscd,amsthm}
\usepackage[margin=1.5in]{geometry}
\usepackage{graphicx}

\usepackage{verbatim}

\numberwithin{equation}{section}

\newtheorem{lemma}[equation]{Lemma}
\newtheorem{theorem}[equation]{Theorem}
\newtheorem{conjecture}[equation]{Conjecture}

\newtheorem{proposition}[equation]{Proposition}

\theoremstyle{remark}

\newtheorem{remark}[equation]{Remark}

\DeclareMathOperator{\fo}{{\mathfrak o}}

\newcommand{\abs}[1]{\lvert #1\rvert}
\renewcommand{\bar}[1]{#1\llap{$\overline{\phantom{\rm#1}}$}}

\begin{document}



\title{Uniform boundedness of $S$-units in arithmetic dynamics}

\author{H.~Krieger}
\address{
Holly Krieger\\
Department of Mathematics\\
Massachusetts Institute of Technology\\
77 Massachusetts Avenue\\
Cambridge, MA 02139\\
USA
}
\email{hkrieger@math.mit.edu}

\author{A.~Levin}
\address{
Aaron Levin
Department of Mathematics\\
Michigan State University\\
619 Red Cedar Road\\
East Lansing, MI 48824\\
USA
}
\email{adlevin@math.msu.edu}

\author{Z.~Scherr}
\address{
Zachary Scherr
Department of Mathematics\\
University of Pennsylvania\\
David Rittenhouse Lab\\
209 South 33rd Street\\
Philadelphia, PA 19104-6395\\
USA
}
\email{zscherr@math.upenn.edu}

\author{T.~J.~Tucker}
\address{
Thomas Tucker\\
Department of Mathematics\\
University of Rochester\\
Rochester, NY 14627\\
USA
}
\email{thomas.tucker@rochester.edu}

\author{Y.~Yasufuku}
\address{
Yu Yasufuku\\
Department of Mathematics\\
College of Science and Technology\\
Nihon University\\
1-8-14 Kanda-Surugadai\\
Chiyoda-ku Tokyo 101-8308\\
Japan
}
\email{yasufuku@math.cst.nihon-u.ac.jp}

\author{M.~E. Zieve}
\address{
Michael Zieve\\
  Department of Mathematics\\
  University of Michigan\\
  Ann Arbor, MI 48109--1043\\
  USA
}

\email{zieve@umich.edu}

\thanks{The authors were partially supported by NSF grants
  DMS-1303770 (H.K.), DMS-1102563 (A.L.),
  DMS-1200749 (T.T.), and DMS-1162181 (M.Z.). The fifth author was partially supported by JSPS Grants-in-Aid 23740033.}


 \begin{abstract} Let $K$ be a number field and let $S$ be a finite set of places of $K$ which contains all the Archimedean places.
 For any $\phi(z)\in K(z)$ of degree $d\ge 2$ which is not a $d$-th power in $\bar{K}(z)$, Siegel's theorem implies that
the image set $\phi(K)$ contains only finitely many $S$-units.  We conjecture that
 the number of such $S$-units is bounded by a function of $\abs{S}$ and $d$ (independently of $K$ and $\phi$).  We prove this conjecture for several classes of rational functions, and show that the full conjecture follows from the Bombieri--Lang conjecture.
 \end{abstract}


\maketitle

\section{Introduction}
\label{intro}

Let $K$ be a number field, let $S$ be a finite set of places of $K$ which contains the set $S_{\infty}$
of Archimedean places of $K$, and write $\fo_S$ for the ring of $S$-integers of $K$ and
$\fo^*_S$ for the group of $S$-units of $K$.  The genus-$0$ case of Siegel's theorem asserts that,
for any $\phi(z)\in K(z)$ which has at least three poles in $\mathbb{P}^1(\bar K)$, the image set
$\phi(K)$ contains only finitely many $S$-integers.  However, the number of $S$-integers in $\phi(K)$
cannot be bounded independently of $\phi(z)$, even if we restrict to functions $\phi(z)$ having a fixed degree,
since $\psi(z):=\beta\phi(z/\beta)$ satisfies $\psi(K)=\beta\phi(K)$ for any $\beta\in K^*$.

Although the number of $S$-integers in $\phi(K)$ cannot be bounded in terms of only $K$, $S$,
and $\deg(\phi)$, such a bound may be possible for the number of $S$-units in $\phi(K)$.
In fact we conjecture that there is a bound depending only on $\abs{S}$ and $\deg(\phi)$ (and not on $K$):

\begin{conjecture} \label{generalmainconjecture} For any integers $s\ge 1$ and $d\ge 2$, there is a constant
$C = C(s, d)$ such that for any
\begin{itemize}
\item  number field $K$
\item $s$-element set $S$ of places of $K$ with $S\supseteq S_\infty$
\item degree-$d$ rational function $\phi(z) \in K(z)$ which is not a $d$-th power in $\bar{K}(z)$
\end{itemize}
%
%
we have
\[
\abs{\phi(K)\cap\fo^*_S}\le C.
  \]
\end{conjecture}

We will prove Conjecture~\ref{generalmainconjecture} in case $\phi(z)$ is restricted to certain classes of
rational functions, and we will also prove that the full conjecture is a consequence of a variant of the Caporaso--Harris--Mazur
conjecture on uniform boundedness of rational points on curves of fixed genus.

We also consider a variant of Conjecture~\ref{generalmainconjecture}, which addresses $S$-units in an orbit of $\phi$
rather than in the image set $\phi(K)$.  Here, for any $\alpha\in\mathbb{P}^1(K)$, the \emph{orbit} of $\alpha$
under $\phi(z)$ is the set
\[
 \mathcal{O}_{\phi}(\alpha) := \{ \phi^n(\alpha): n \geq 1 \},
 \]
where $\phi^n(z) = \phi \circ \cdots \circ \phi$ denotes the $n$-fold composition of $\phi$ with itself.
For any $\phi(z)\in K(z)$ of degree at least $2$ such that $\phi^2(z)\notin K[z]$, Silverman~\cite{Silverman}
showed that $\mathcal{O}_{\phi}(\alpha)\cap\fo_S$ is finite.
However, as above, the size of this intersection cannot be bounded in terms of $K$, $S$, and $\deg(\phi)$.
We conjecture that there is a uniform bound on the number of $S$-units in an orbit:

\begin{conjecture} \label{mainconjecture} For any integers $s\ge 1$ and $d\ge 2$, there is a constant
$C = C(s, d)$ such that for any
\begin{itemize}
\item  number field $K$
\item $s$-element set $S$ of places of $K$ with $S\supseteq S_\infty$
\item degree-$d$ rational function $\phi(z) \in K(z)$ which is not of the form $\beta z^{\pm d}$ with $\beta\in K^*$
\item $\alpha \in \mathbb{P}^1(K)$
\end{itemize}
%
%
we have
\[
 \abs{\mathcal{O}_{\phi}(\alpha) \cap \fo^*_S} \le C.
 \]
\end{conjecture}


\begin{remark}\label{conjcomp}
We note that Conjecture~\ref{mainconjecture} follows from Conjecture~\ref{generalmainconjecture}.
For, if $\phi(z)\ne\beta z^{\pm d}$ then $\phi^2(z)$ has a total of at least three zeroes and poles by Lemma~\ref{poles},
and hence is not
a $d^2$-th power in $\bar K(z)$.  Thus Conjecture~\ref{generalmainconjecture} implies that $\abs{\phi^2(K)\cap\fo_S^*}\le C(s,d)$, so that $\abs{\mathcal{O}_{\phi}(\alpha)\cap\fo_S^*}\le C(s,d)+1$.
\end{remark}

\begin{remark}\label{bounddeg}
The hypotheses of Conjectures~\ref{generalmainconjecture} and \ref{mainconjecture} imply that $[K:\mathbb{Q}]\le 2s$, since $S_{\infty}\subseteq S$.
\end{remark}

In Section~\ref{evidence} we prove the following results, which
show that Conjectures~\ref{generalmainconjecture} and \ref{mainconjecture}
 would be true if we allowed the constants $C$ in the conjectures to depend on $S$ and $\phi$ rather than
just $s$ and $d$.  These results also indicate the special behavior of the functions excluded in the statements of the
conjectures.

\begin{proposition} \label{generalfiniteness} Let $K$ be a number field, let $S$ be a finite set of places of $K$ with $S\supseteq S_{\infty}$,
and let $\phi(z)\in K(z)$ be any rational function.  If $\abs{\phi^{-1}(\{0,\infty\})}\ne 2$ then $\phi(K)\cap\fo_S^*$ is finite.
If $\abs{\phi^{-1}(\{0,\infty\})}=2$ then there is a finite set $S'\supseteq S$
for which $\phi(K)\cap\fo_{S'}^*$ is infinite.
\end{proposition}

\begin{proposition} \label{finiteness}
Let $K$ be a number field, let $S$ be a finite set of places of $K$ with $S\supseteq S_{\infty}$, and let $\phi(z)\in K(z)$
have degree $d\ge 2$.
If $\phi(z)$ does not have the form $\beta z^{\pm d}$
with $\beta\in K^*$, then
there is a constant $C(S,\phi)$ such that every $\alpha\in\mathbb{P}^1(K)$ satisfies
$\abs{\mathcal{O}_{\phi}(\alpha)\cap\fo^*_S} \le C(S,\phi)$.
Conversely, if $\phi(z)=\beta z^{\pm d}$ with $\beta\in K^*$  then there exist $\alpha\in\mathbb{P}^1(K)$ and
a finite set $S\supseteq S_{\infty}$ for which
$\mathcal{O}_{\phi}(\alpha)\cap\fo^*_S$ is infinite.
\end{proposition}

We note that the hard portions of these propositions are immediate consequences of Siegel's theorem.
For, if $\abs{\phi^{-1}(\{0,\infty\})}>2$ then $\psi(z):=\phi(z)+1/\phi(z)$ has at least three poles so that $\psi(K)\cap\fo_S$ is
finite; but $\psi(\beta)$ is in $\fo_S$
whenever $\phi(\beta)$ is in $\fo_S^*$, so also $\phi(K)\cap\fo_S^*$ is finite.
  Next, if $\phi^{-1}(\{0,\infty\})$ is a two-element set other than $\{0,\infty\}$, then
Lemma~\ref{poles} implies that $\abs{\phi^{-2}(\{0,\infty\})}>2$, so that $\phi^2(K)\cap\fo_S^*$ has size $N<\infty$, whence
$\abs{\mathcal{O}_{\phi}(\alpha)\cap\fo_S^*}\le N+1=C(S,\phi)$.

In Section \ref{results} we prove Conjectures~\ref{generalmainconjecture} and \ref{mainconjecture} for some families of polynomial maps.
The first family consists of monic polynomials in $\fo_S[z]$:

\begin{theorem} \label{monicpoly} Let $s\ge 1$ and $d\ge 2$ be integers.  There is a constant $C = C(s, d)$ such that
for any
\begin{itemize}
\item  number field $K$
\item $s$-element set $S$ of places of $K$ with $S\supseteq S_\infty$
\item degree-$d$ monic polynomial $\phi(z) \in \fo_S[z]$ which does not equal $(z-\beta)^d$ for any $\beta\in K$
\end{itemize}
%
%
we have
\[
\abs{\phi(K)\cap \fo^*_S} \le C.
\]
\end{theorem}

Theorem~\ref{monicpoly} proves Conjecture~\ref{generalmainconjecture} for monic polynomials in $\fo_S[z]$; for such
 polynomials, Conjecture~\ref{mainconjecture} follows by applying Theorem~\ref{monicpoly} to $\phi^2(z)$.

We also prove Conjecture~\ref{mainconjecture} for monic polynomials in $K[z]$ in which the coefficients of all but one term are
in $\fo_S$, so long as this exceptional term does not have degree $d-1$.

\begin{theorem} \label{unicritical} Let $K$ be a number field, and let $S$ be a finite set of places of $K$ with
$S\supseteq S_{\infty}$.  For any monic $\phi_0(z)\in\fo_S[z]$, and any $\beta\in K\setminus\fo_S$ and
$0\le i<\deg(\phi_0)-1$, the polynomial $\phi(z):=\phi_0(z)+\beta z^i$ satisfies
\[
\abs{\mathcal{O}_{\phi}(\alpha)\cap  \fo^*_S} \le 1
\]
for any $\alpha\in K$.
\end{theorem}

Conjecture~\ref{mainconjecture} follows from \cite[Thm.~2]{Levin} for rational functions of the form
\[
\phi(z):=\frac{z^d+\beta_{d-1}z^{d-1}+\cdots+\beta_1z}{\gamma_{d-1}z^{d-1}+\gamma_{d-2}z^{d-2}+\cdots+\gamma_1z+1}
\]
with $\beta_1,\dots,\beta_{d-1},\gamma_1,\dots,\gamma_{d-1}\in\fo_S$ and $\phi(z)\ne z^d$.
For, \cite[Thm.~2]{Levin} gives a uniform bound on the number of elements of $K$ in the backwards orbit of
any element of $\fo^*_S$.  This also bounds the number of $S$-units in $\mathcal{O}_{\phi}(\alpha)$ for any $\alpha\in K$, since if
$\phi^n(\alpha)\in\fo^*_S$ then $\alpha,\phi(\alpha),\dots,\phi^{n-1}(\alpha)$ are elements of $K$ in the backwards
orbit of $\phi^n(\alpha)$.

We prove our conjectures for some further classes of rational functions in Section~\ref{extra}.

In Section~\ref{evidence} we show that our conjectures are
consequences of the following variant of the deep conjecture of
Caporaso--Harris--Mazur \cite{CHM} concerning rational points
 on curves of a fixed genus.

\begin{conjecture} \label{CHM} Fix integers $g \geq 2$ and $D \geq 1$.  There is a constant $N = N(D, g)$ such that
$\abs{X(K)}\le N$ for every smooth, projective, geometrically irreducible genus-$g$ curve
$X$ defined over a degree-$D$ number field $K$.
\end{conjecture}

\begin{theorem} \label{generalfollowsfromCHM} If Conjecture~\ref{CHM} is true then Conjecture~\ref{generalmainconjecture}
and Conjecture~\ref{mainconjecture} are true.
\end{theorem}

\begin{remark}
Conjecture~\ref{CHM} follows from the Bombieri--Lang conjecture~\cite{Pacelli}.
\end{remark}

We thank ICERM and the organizers of the 2012 ICERM workshop on Global Arithmetic Dynamics, where collaboration for this project began.

\section{Special classes of rational functions} \label{results}

In this section we prove Theorems~\ref{monicpoly} and \ref{unicritical}.

\begin{proof} [Proof of Theorem~\ref{monicpoly}] Let $K$ be a number
  field, let $S$ be a finite set of places of $K$ with $S\supseteq S_{\infty}$, and let
  $\phi(z)\in\fo_S[z]$ be monic of degree $d\ge 2$ with $\phi(z)\ne (z-\beta)^d$ for any $\beta\in K$.
Then $\phi(z)$ has at least two distinct roots $\delta_1,\delta_2$ in $\bar K$.
Let $K'=K(\delta_1,\delta_2)$ and let $S'$ be the set of places of $K'$ which lie over places in $S$, so that
$\abs{S'}\le [K':K]\abs{S}\le d(d-1)\abs{S}$ and $\delta_i\in\fo_{S'}$.
Then we can write
\[
\phi(z) = (z - \delta_1)(z-\delta_2) \psi(z),
\]
where $\psi(z)$ is a monic polynomial in $\fo_{S'}[z]$.
Let $\gamma\in K$ satisfy $\phi(\gamma)\in\fo_S^*$.  Then we must have $\gamma\in\fo_S$, so that
both $u_i := \gamma - \delta_i$ and $\psi(\gamma)$ are in $\fo_{S'}$.  Since $u_1u_2\psi(\gamma)=\phi(\gamma)$
is in $\fo_S^*$, it follows that $u_1,u_2\in\fo_{S'}^*$.  In addition we have
\begin{equation} \label{uniteqn}
\frac 1{\delta_2 - \delta_1} u_1 - \frac 1{\delta_2 - \delta_1} u_2 = 1.
\end{equation}
Moreover, $\gamma$ is uniquely determined by $u_1$, so the number of elements $\gamma\in\fo_{S}$ for which
$\phi(\gamma)\in\fo^*_{S}$ is at most the number of solutions to (\ref{uniteqn}) in elements $u_1,u_2\in\fo^*_{S'}$.
Finally, it is known that the number of such solutions is at most $C_1 C_2^{\abs{S'}-1}$ for some absolute
constants $C_1, C_2$
\cite{Evertse} (in fact, we can take $C_1=C_2=256$ \cite{Beukers-Sch}).  Therefore
$\abs{\phi(K)\cap\fo_S^*}$ is bounded by a function of $\abs{S'}$, and hence by a function of $\abs{S}$ and $d$.
\end{proof}


\begin{proof} [Proof of Theorem~\ref{unicritical}]
Let $v\notin S$ be a non-Archimedean place of $S$ such that $\abs{\beta}_v > 1$.
  Suppose that ${\mathcal O}_{\phi}(\alpha)$ contains an $S$-unit, and let $m$ be the least non-negative integer
  for which $\phi^m(\alpha)\in\fo_S^*$.  Writing $\gamma:=\phi^m(\alpha)$, we will show by induction that
  $\abs{\phi^n(\gamma)}_v =
  \abs{\beta}_v^{d^{n-1}}$ for every $n\ge 1$. The strong triangle inequality implies that
$
\abs{\phi(\gamma)}_v = \abs{\beta}_v,
$
proving the base case $n=1$.  If $\delta:=\phi^n(\gamma)$ satisfies
$\abs{\delta}_v = \abs{\beta}_v^{d^{n-1}}$ for some $n\ge 1$, then
$\abs{\phi_0(\delta)}_v=\abs{\beta}_v^{d^n}$ and $\abs{\beta\delta^i}_v=\abs{\beta}_v^{1+id^{n-1}}$;
our hypothesis $i<d-1$ implies that $d^n>1+id^{n-1}$, so that $\abs{\phi^{n+1}(\gamma)}_v=\abs{\beta}_v^{d^n}$,
which completes the induction.  It follows that $\phi^n(\gamma)\notin\fo_S$ for every $n>0$, so that
${\mathcal O}_{\phi}(\alpha)$ contains exactly one $S$-unit, which concludes the proof.
\end{proof}

\section{Connection with rational points on curves}
\label{evidence}

In this section we prove Theorem~\ref{generalfollowsfromCHM} and
Propositions~\ref{generalfiniteness} and \ref{finiteness}.
We begin by relating $S$-units in an orbit to rational points on certain curves.

\begin{lemma} \label{curves}
Let $K$ be a number field, let $S$ be a finite set of places of $K$ with $S\supseteq S_{\infty}$, and
let $\psi(z)\in K(z)$ be a nonconstant rational function.
For any prime $p$ with $p>\deg(\psi)$, there are elements $\gamma_1,\dots,\gamma_t\in\fo_S^*$, where $t\le p^{\abs{S}}$,
with the following properties:
\begin{itemize}
\item for each $i$, the affine curve $X_i$ defined by $y^p=\gamma_i\psi(z)$ is geometrically irreducible
\item we have $\abs{\psi(K)\cap\fo_S^*}\le\sum_{i=1}^t N_i$ where
$N_i$ is the number of points in $X_i(K)$ having nonzero $y$-coordinate.
\end{itemize}
\end{lemma}

\begin{proof}
First note that $y^p=\gamma\psi(z)$ is geometrically irreducible for any $\gamma\in K^*$, since $\gamma\psi(z)$
is not a $p$-th power in $\bar{K}(z)$.
Dirichlet's $S$-unit theorem asserts that $\fo_S^*\cong\mu_K\times\mathbb{Z}^{\abs{S}-1}$, where
$\mu_K$ denotes the group of roots of unity in $K$.  Since $\mu_K$ is cyclic, it follows that
$\fo_S^*/(\fo_S^*)^p\cong (\mathbb{Z}/p\mathbb{Z})^r$ where $r\in\{\abs{S}-1,\abs{S}\}$.
Let $\Gamma$ be a set of $p^r$ elements in $\fo_S^*$ whose images in $\fo_S^*/(\fo_S^*)^p$ are pairwise distinct.
For any $\beta\in K$ such that $\psi(\beta)\in\fo_S^*$, we can write $\psi(\beta)=\gamma^{-1}\delta^p$ for
some $\gamma\in\Gamma$ and $\delta\in\fo_S^*$.  Then $(\delta,\beta)$ is a $K$-rational point on the curve
$y^p=\gamma\psi(z)$ whose $y$-coordinate is nonzero.  Since the $z$-coordinate of this point is $\beta$, the
result follows.
\end{proof}

In order to control the number $N_i$ from Lemma~\ref{curves}, we must control the genus of the curve $X_i$.
This computation is classical: the genus is $(p-1)(m-2)/2$ where $m$ is the total number of points in $\mathbb{P}^1(\bar K)$
which are either zeroes or poles of $\psi(z)$.  In particular, if $m=2$ then the genus is $0$, in which case $N_i$ can
be infinite.  We avoid this difficulty by applying the above result with $\psi(z)$ being the second iterate $\phi^2(z)$ of a given
function $\phi(z)$, so that $\psi(z)$ has a combined total of
at least three zeroes and poles by the following lemma.

\begin{lemma} \label{poles}
Let $\phi(z)\in \mathbb{C}(z)$ be any rational function of degree $d\ge 2$ which is not of the form
$\beta z^{\pm d}$ with $\beta\in \mathbb{C}^*$.  Then $\abs{\phi^{-2}(\{0,\infty\})}\ge 3$.
\end{lemma}

\begin{proof}
Write $m:=\abs{\phi^{-2}(\{0,\infty\})}$, so we must show that $m\ge 3$.
Plainly $m\ge\abs{\phi^{-1}(\{0,\infty\})}\ge 2$, so the conclusion holds
unless $\abs{\phi^{-1}(\{0,\infty\})}=2$.  In this case $\phi$ is totally ramified over both $0$ and $\infty$, so the
Riemann--Hurwitz formula (or writing down $\phi(z)$) implies that $\phi$ is unramified over all
other points.  Since $\phi(z)$ does not have the form $\beta z^{\pm d}$, we know that $\phi^{-1}(\{0,\infty\})\ne\{0,\infty\}$,
so that at least one point in $\phi^{-1}(\{0,\infty\})$ has $d$ distinct $\phi$-preimages.
Since each point has at least one preimage, we conclude that $m\ge d+1\ge 3$, as desired.
\end{proof}

We now prove Theorem~\ref{generalfollowsfromCHM}.

\begin{proof}[Proof of Theorem~\ref{generalfollowsfromCHM}]
Let $K$ be a number field, let $S$ be a finite set of places of $K$ with $S\supseteq S_{\infty}$, and let $\phi(z)\in K(z)$
have degree $d\ge 2$ where $m:=\abs{\phi^{-1}(\{0,\infty\})}$ is at least $3$.
Let $p$ be the smallest prime for which $p>d$ and $(p-1)(m-2)>2$.  Then $p=5$ if $d=2$ and $m=3$, and in all other
cases $p<2d$ by Bertrand's Postulate.
Let $\gamma_1,\dots,\gamma_t$ satisfy the conclusion of
Lemma~\ref{curves}, so that $\gamma_i\in K^*$ and $t\le p^{\abs{S}}$.
Writing $X_i$ for the curve $y^p=\gamma_i\phi(z)$, and $N_i$ for the number of points in $X_i(K)$ having
nonzero $y$-coordinate, it follows that $\abs{\phi(K)\cap\fo_S^*}\le\sum_{i=1}^t N_i$.
Since every point on $X_i$ having nonzero $y$-coordinate is nonsingular, we see that $N_i$ is bounded above by
the number of $K$-rational points on the unique smooth projective curve $Y_i$ over $K$ which is birational to $X_i$.
The genus $g$ of $X_i$ equals $(p-1)(m-2)/2$, so our choice of $p$ ensures that
\[
2\le g\le \frac{(\frac52d-1)(2d-2)}2.
\]
If Conjecture~\ref{CHM} is true then $\abs{Y_i(K)}$ is bounded by a constant which depends only on
the genus of $Y_i(K)$ and the degree $[K:\mathbb{Q}]$.  Since the genus is bounded by a function of $d$, and
the degree $[K:\mathbb{Q}]$ is bounded by a function of $\abs{S}$ (by Remark~\ref{bounddeg}), it follows that
$\abs{Y_i(K)}$ is bounded by a constant depending on $d$ and $\abs{S}$.  Since $t\le p^{\abs{S}}\le (5d/2)^{\abs{S}}$,
this proves that Conjecture~\ref{CHM} implies Conjecture~\ref{generalmainconjecture}, and Conjecture~\ref{mainconjecture} follows by applying Conjecture~\ref{generalmainconjecture} to $\phi^2(z)$ in light of Lemma~\ref{poles}.

\end{proof}

The first (and hardest) assertion in Proposition~\ref{generalfiniteness} follows from the above proof, by using
Faltings' theorem \cite{Faltings} instead of Conjecture~\ref{CHM}.
We now give a more elementary proof of Proposition~\ref{generalfiniteness}.

\begin{proof}[Proof of Proposition~\ref{generalfiniteness}]
If $\abs{\phi^{-1}(\{0,\infty\})}>2$ then the function $\psi(z):=\phi(z)+1/\phi(z)$ satisfies $\abs{\psi^{-1}(\{0,\infty\})}\ge 3$, so
$\psi(K)\cap\fo_S$ is finite by Siegel's theorem; but $\psi(\beta)$ is in $\fo_S$ whenever $\phi(\beta)$ is in $\fo_S^*$, so it follows that
$\phi(K)\cap\fo_S^*$ is finite.  Now assume that $\abs{\phi^{-1}(\{0,\infty\})}=2$, so that
$\phi(z)=\gamma\mu(z)^d$ for some $d\ge 1$, some $\gamma\in K^*$, and some
degree-one $\mu(z)\in K(z)$.  Let $S'$ be a finite set of places of $K$ such that $\gamma\in\fo_{S'}^*$, $S'\supseteq S$, and $\abs{S'}>1$.
Since $\mu(K)$ contains all but at most one element
of $K$, it follows that $\phi(K)$ contains all but at most one element of
$\gamma (\fo_{S'}^*)^d$.  Since $\gamma\in\fo_{S'}^*$ and $\abs{S'}>1$, this shows
that $\phi(K)\cap\fo_{S'}^*$ is infinite.
\end{proof}

We conclude this section by proving Proposition~\ref{finiteness}.

\begin{proof}[Proof of Proposition~\ref{finiteness}]
If $\phi(z)$ does not have the form $\beta z^{\pm d}$ then $\abs{\phi^{-2}(\{0,\infty\})}\ge 3$ by Lemma~\ref{poles},
so Proposition~\ref{generalfiniteness} implies that $\phi^2(K)\cap\fo_S^*$ has size $N<\infty$, whence
$\abs{\mathcal{O}_{\phi}(\alpha)\cap\fo_S^*}\le N+1=C(S,\phi)$.  Now consider $\phi(z)=\beta z^{\pm d}$ with $\beta\in K^*$
and $d\ge 2$.  Any $\alpha\in K^*$ satisfies
$\mathcal{O}_{\phi}(\alpha)\subseteq\fo_S^*$ where $S$ is the union of $S_{\infty}$ with the set of places $v$ of $K$
for which $\abs{\alpha}_v\ne 1$ or $\abs{\beta}_v\ne 1$.  If $\alpha\in K^*$ is not a root of unity then $\mathcal{O}_{\phi}(\alpha)$
is infinite, so that $\mathcal{O}_{\phi}(\alpha)\cap\fo_S^*$ is infinite.
\end{proof}

\section{Additional Remarks} \label{extra}

We make two additional remarks.  First, the proofs of Theorems \ref{monicpoly} and \ref{unicritical} can be modified
to treat some classes of Laurent polynomials.  For example, let $d$ and $d'$ be distinct positive integers, and let
$\phi(z)=(\gamma_dz^d+\dots+\gamma_1z+\gamma_0)/z^{d'}$ where $\gamma_i\in\fo_S$ and $\gamma_d,\gamma_0\in\fo_S^*$.  Suppose in addition that the numerator is not a $d$-th power in $\bar{K}[z]$.
Then $\abs{\phi(K)\cap\fo^*_S} \le C(s,d)$ for any $\alpha\in\mathbb P^1(K)$.
  Indeed, since $\gamma_0$ and $\gamma_d$ are assumed to be units, $\phi(\beta)$ cannot be in $\fo_S^*$ if $\abs{\beta}_v \neq 1$ for some $v\notin S$.
Thus we need only consider $\beta\in\fo_S^*$, and now the desired bound follows from the proof of
Theorem~\ref{monicpoly}.

As another example, consider $\phi(z)=(\gamma_d z^d+\dots+\gamma_1 z+\gamma_0)/z^{d'}$ where $d>d'$, $\gamma_i\in K$, and there is
some $v\notin S$ for which $\abs{\gamma_d}_v>\max(1,\abs{\gamma_i}_v)$ for each $i<d$.   Then
$\abs{\mathcal{O}_{\phi}(\alpha)\cap\fo^*_S} \le 1$ for any $\alpha \in \mathbb P^1(K)$, as the orbit of an $S$-unit cannot contain another $S$-integer by the proof of Theorem~\ref{unicritical}.
 Both this class of examples and the previous class are quite special,
 but they serve as further evidence for Conjectures~\ref{generalmainconjecture} and \ref{mainconjecture}.

We conclude by noting that the constant $C$ in Conjectures~\ref{generalmainconjecture} and
\ref{mainconjecture} must depend on both $s$ and $d$.
The necessity of
dependence on $s$ is clear.  Dependence on $d$ is also required, since by Lagrange interpolation one can construct
polynomials $\phi(z)\in K[z]$ in which the first several $\phi^i(\alpha)$ take on any prescribed distinct values in $K$
while also $\phi(z)$ has at least two zeroes (and hence is not $\beta z^{\pm d}$).

\end{document}